\documentclass[12pt,oneside]{amsart}
\usepackage{amsmath}
\usepackage{amsthm}
\usepackage{amsfonts}
\usepackage{amssymb}
\usepackage{enumerate}
\usepackage{comment}
\usepackage{xcolor}
\usepackage{soul}
\usepackage[normalem]{ulem}
\usepackage{cancel}
\newtheorem{theorem}{Theorem}
\newtheorem{lemma}[theorem]{Lemma}
\newtheorem{proposition}[theorem]{Proposition}
\newtheorem{corollary}[theorem]{Corollary}

\newtheorem{prop}{Proposition}
\newtheorem{thm}{Theorem}

\theoremstyle{definition}

\newtheorem{example}[theorem]{Example}

\theoremstyle{remark}
\newtheorem{remark}[theorem]{Remark}
\newcommand{\DD}{{\mathbb D}}

\DeclareMathOperator{\Aut}{Aut} 
 
 \DeclareMathOperator{\re}{Re}

\DeclareMathOperator{\im}{Im}

\renewcommand{\phi}{\varphi}


\begin{document}

\title{Isometries in the diamond}


\address{Faculty of Mathematics and Computer Science, Jagiellonian
University,  \L ojasiewicza 6, 30-348 Krak\'ow, Poland}

\author{Anand Chavan}\email{anand.chavan@doctoral@uj.edu.pl}
\author{W\l odzimierz Zwonek}\email{wlodzimierz.zwonek@uj.edu.pl}

\thanks{Supported by the Preludium bis grant no. 2021/43/O/ST1/02111 of the National Science Centre,
Poland.}


\begin{abstract} We show the (anti)holomorphicity of smooth Kobayashi isometries of the diamond, the domain defined as $\triangle:=\{z\in\mathbb C^2:|z_1|+|z_2|<1\}$. Additionally, we discuss the problem of uniqueness of real geodesics, left inverses and strict convexity of indicatrices.
\end{abstract}
\maketitle

\section{Introduction}
\subsection{Results} The problem whether the isometries (with respect to the Kobayashi distance/metric or more generally with respect to some holomorphically invariant distances or metrics) are (anti)holomorphic has been recently investigated very thoroughly. The most significant tool in this context is the Lempert Theory. And in the class of convex domains, where the Lempert Theorem holds, the most general result is due to Edigarian (see \cite{Edi 2019}) that gives the positive answer to this question for strictly convex domains. 

The trivial example of the bidisc shows that the positive answer is not possible without additional assumption. It is natural that we may try to study the problem for specific convex domains that are not strictly convex. In this context it is quite natural to consider the domain $\triangle$, called {\it diamond}, that is defined by the formula $\triangle:=\{z\in\mathbb C^2:|z_1|+|z_2|<1\}$.

The positive result in that direction paper is the following.

\begin{thm}\label{theorem:isometry} 
Let $F:\triangle\to\triangle$ be a $C^1$-smooth Kobayashi isometry. Then $F$ is holomorphic or antiholomorphic, in other words, up to a permutation of variables it is of the form $F(z)=(\omega_1z_1,\omega_2z_2)$, $z\in\triangle$ or $F(z)=(\omega_1\overline{z_1},\omega_2\overline{z_2})$, $z\in\triangle$ for some $|\omega_j|=1$, $j=1,2$.
\end{thm}
Recall that the complete description of holomorphic automorphisms of $\triangle$ goes back to N. Kritikos (see \cite{Kri 1927}) so the above theorem may be seen as a direct generalization of this classical result.

Though we give the ad hoc proof of the above theorem we utilize some results on strict convexity of indicatrices, possible forms of real geodesics or the property of uniqueness of right inverses that are formulated in a more general context and may well be used in other situations. These results of a more general character are presented in Section~\ref{section:uniqueness-geodesics}. In particular, a close relation between the uniqueness of real geodesics and property of the uniqueness of right inverses and strict convexity of indicatrices is presented there. 

At this place let us also draw Reader's attention to a very recent result of Edigarian that states that Kobayashi isometries of the symmetrized bidisc $\mathbb G_2$ are holomorphic (or antiholomorphic), too. The symmetrized bidisc though not convex (and even not biholomorphic to a convex domain) satisfies the Lempert Theorem and additionally the indicatrix at $(0,0)$ is (up to linear isomorphism) the domain $\triangle$ that is the indicatrix of $\triangle$ at $(0,0)$, too (see \cite{Agl-You 2001}, \cite{Cos 2004}, \cite{Agl-You 2004}). 

\subsection{Motivation and background} 
The general problem that may be seen as the motivation for our considerations is the following. Consider $F:M\to N$ with $M,N$ being Kobayashi hyperbolic manifolds being the isometry with respect to the Kobayashi metric or distance. Under which assumptions $F$ is necessarily holomorphic or antiholomorphic? The more concrete problems may require additionally that $M$ and $N$ have the same dimension or even $M=N$.

As already mentioned in the case of convex domains the most general positive result goes back to Edigarian (see \cite{Edi 2019}) who gave the positive answer in the case of strictly convex domains. Edigarian extended earlier results (see for instance \cite{Gau-Ses 2013}). The first positive results in this direction could be found in \cite{Kuc-Ray 1988}, \cite{Zwo 1993}, \cite{Zwo 1995}. More recently the case of Cartesian products/bounded symmetric domains/strongly pseudoconvex domains was studied among others in \cite{Lemm 2022}, \cite{Ant 2017}, \cite{Kim-Seo 2022}, \cite{Mah 2012}, \cite{Ses-Ver 2006}. A quite general conjecture in this context is Conjecture 5.2 from \cite{Gau-Ses 2017} that states that the Kobayashi isometries are (anti)holomorphic provided the manifolds $M$ and $N$ are not Cartesian products. Finally, the most recent result is contained in \cite{Edi 2024} where the studied domain is the symmetrized bidisc.

Apart from the rigidity result for the Kobayashi isometries of $\triangle$ we present in the paper some auxiliary results that are formulated in a more general setting and that refer to the uniqueness of right and left inverses and strict convexity of indicatrices. Note that the latter could play some role in a better understanding of the theory of extension sets admitting norm preserving extensions of bounded holomorphic functions. The complete description of extension sets was first settled for the bidisc in \cite{Agl-McC 2003} and was later developed.  It should be mentioned that in the symmetrized bidisc one could find the first extension set being not a retract (see \cite{Agl-Lyk-You 2019}). Similar phenomenon was recently observed in the domain $\triangle$ (see \cite{Agl-Kos-McC 2023}). That indicates that the better understanding of the domain $\triangle$ could help in the further study of the theory of extension sets, too. 

\section{Lempert domains. Complex and real geodesics. Left inverses. Indicatrices.}\label{section:uniqueness-geodesics} 
Though some of the results presented below will be true in a more general context we mostly reduce to the case when the domains under considerations satisfy the Lempert Theorem. We need therefore to introduce suitable definitions and notions. For a domain $D\subset\mathbb C^n$ we define
holomorphically invariant functions. Let $w,z\in D$.

{\it The Carath\'eodory (pseudo)distance} is defined as follows 
$$c_D(z,w):= \sup \{p(F(z),F(w)): F \in \mathcal{O}(D,\mathbb{D})\}$$
    where $p$ is the Poincar\'e distance on the unit disc $\mathbb D\subset\mathbb C$. 

{\it The Lempert function} is defined as
    $$l_D(z,w):=\inf\{p(\sigma, \zeta): f(\sigma)=z \text{ and } f(\zeta)=w \text{ where }f\in\mathcal{O}(\mathbb{D},D)\}.$$

    {\it The Kobayashi (pseudo)distance $k_D$} is the largest (pseudo)distance not greater than Lempert function. We have the inequalities
    $c_D\leq k_D\leq l_D$.

On the other side, for $p\in D $ and $X \in \mathbb{C}^n$ we define {\it the Caratho\'edory (pseudo)metric} as:
    $$\gamma_D(p;X):=\sup\{|F'(p) X|: F\in \mathcal{O}(D,\mathbb{D}) \text{ and } F(p)=0  \}.$$

{\it The Kobayashi (pseudo)metric} is defined as:
    $$\kappa_{D}(p;X):=\inf\{|\alpha|: \alpha f'(0)=X \text{ for } f \in \mathcal{O}({\mathbb{D},D})\text{ and } f(0)=p\}.$$
    Recall that $\gamma_D\leq \kappa_D$.
    
 We call the domain $D\subset\mathbb C^n$ {\it Lempert domain} if $D$ is taut and the Lempert Theorem holds on $D$, i. e. if we have the equalities
\begin{equation}
    l_D\equiv k_D\equiv c_D,\; \kappa_D\equiv \gamma_D.
\end{equation}

Recall that a holomorphic mapping $f:\mathbb D\to D$ is called a {\it complex geodesic } if there is a holomorphic function $F:D\to\mathbb D$ such that $F\circ f$ is an automorphism. We call such an $F$ the {\it left inverse to $f$}. Composing the function $F$ with an automorphism of $\mathbb D$ we may, without loss of generality, additionally assume that $F\circ f$ is the identity - this lets us easier formulate different results on uniqueness. Recall that for any pair of different points $w,z$ in the Lempert domain $D$ there is a complex geodesic passing through them. In the infinitesimal case we know that for any pair $(z,X)\in D\times\mathbb C^n$ there is a complex geodesic $f:\mathbb D\to D$ such that $f(0)=z$ and $X$ is paralel to $f^{\prime}(0)$. 

Consequently, for the Lempert domain $D\subset\mathbb C^n$ we identify complex geodesics $f$ and $g$ by the relation $f=g\circ a$, where $a$ is an automorphism of $\mathbb D$. Similarly, we also identify the left inverses $F,G$ of a given geodesic by the relation $F=a\circ G$ for some automorphism $a$ of $\mathbb D$.

We shall recall basic results that are due to Lempert (see \cite{Lem 1981}, \cite{Lem 1984}) that state that bounded convex domains or strongly linearly convex domains are Lempert domains. More recent results show that the symmetrized bidisc $\mathbb G_2$, the tetrablock $\mathbb E$, or more generally domains $\mathbb L_n$ are Lempert domains, too (see \cite{Cos 2004}, \cite{Agl-You 2004}, \cite{Abo-You-Whi 2007}, \cite{Edi-Kos-Zwo 2013}, \cite{Gho-Zwo 2023}, \cite{Edi 2023}). As to the best reference on that theory, definitions and notations mentioned above we recommend the monograph \cite{Jar-Pfl 2013}. 

\subsection{Real geodesics}
The mapping $\gamma:(-1,1)\to D$ is called a {\it real geodesic} if $p(t,s)=k_D(\gamma(t),\gamma(s))$, $t,s\in(-1,1)$. Recall that sometimes the real geodesics may be understood as obtained from above by suitable composition with diffeomorphisms of intervals. More precisely, we could call $\gamma:\mathbb R\to D$ a real geodesic if $k_D(\gamma(t),\gamma(s))=|t-s|$, $s,t\in\mathbb R$ but in our paper we stick to the first defintion.

Note that if $f:\mathbb D \to D$ is a complex geodesic then $f_{|(-1,1)}$ is a real geodesic.



 Consequently for any Lempert domain $D\subset\mathbb C^n$ and any $p,q\in D$ there is a real geodesic passing through $p$ and $q$ -- the one induced by a complex geodesic.

The existence of (holomorphic) left inverses to geodesics in Lempert domains may be generalized to real geodesics. The result seems to be a folklore nevertheless for the sake of completeness we present its proof below.

\begin{remark} If $D\subset\mathbb C^n$ is a Lempert domain then real geodesics have holomorphic left inverses or if $\gamma:(-1,1)\to D$ is a real geodesic then 
there is a holomorphic function $F:D \to\mathbb D$ such that $F(\gamma(s))=s$, $s\in(-1,1)$. 

In fact, for $t\in(0,1)$ we find a holomorphic $F_t:D\to\mathbb D$ such that $F_t\circ \gamma(\pm t)=\pm t$. Then for $s\in(-t,t)$ we get
\begin{equation}
 p(-t,s)=k_D(\gamma(-t),\gamma(s))\geq c_D(\gamma(-t),\gamma(s))\geq p(-t,F_t(\gamma(s))).
\end{equation}

Similarly we get $p(s,t)\geq p(F_t(\gamma(s)),t)$. But
\begin{equation}
p(-t,t)=p(-t,s)+p(s,t)\geq p(-t,F_t(\gamma(s)))+p(F_t(\gamma(s)),t)\geq p(-t,t),
\end{equation}
which altogether implies that $F_t(\gamma(s))=s$, $-t<s<t$.

Now we make use of the Montel theorem to get a holomorphic limit $F:D\to\mathbb D$ of a subsequence $(F_{t_k})_k$ with $F(\gamma(s))=s$, $s\in(-1,1)$ which finishes the proof.


\end{remark}

\begin{remark} We may prove the following fact: if $D$ is a Lempert domain, $\gamma:[t_1,t_2]\to D$ is a {\it geodesic interval}, i.e. $-1<t_1<t_2<1$ and $k_D(\gamma(t),\gamma(s))=p(t,s)$, $t,s\in [t_1,t_2]$, then $\gamma$ extends to a real geodesic. Actually, let $F$ be a left inverse to a complex geodesic $f$ passing through $\gamma(t_1)$ and$\gamma(t_2)$ and such that $f(t_1)=\gamma(t_1)$, $f(t_2)=\gamma(t_2)$ (consequently, $F(\gamma(t))=t$, $t\in (-1,t_1]\cup [t_2,1)$. Extend the geodesic interval by the formula $\gamma(t):=f(t)$, $t\in(-1,t_1)\cup(t_2,1)$. Then one may easily see that $\gamma$ is a real geodesic. Actually, if $-1<s_1\leq t_1<t_2\leq s_2<1$ then 
\begin{equation}
k_D(\gamma(s_1),\gamma(s_2))=k_D(f(s_1),f(s_2))=p(s_1,s_2).
\end{equation}
On the other hand for $-1<t_1 < s< t_2<1$ we get
\begin{multline}
p(t_1,t_2)=p(F(\gamma(t_1)),F(\gamma(t_2))\leq p(t_1,F(\gamma(s)))+p(F(\gamma(s)),t_2)\leq\\
k_D(\gamma(t_1),\gamma(s))+k_D(\gamma(s),\gamma(t_2))=p(t_1,s)+p(s,t_2)=p(t_1,t_2)
\end{multline}
from which we esily get that $F(\gamma(s))=s$ for all $t_1<s<t_2$ (and so for all $s\in(-1,1)$). Consequently, for $-1<s_1<t_1<s_2<t_2<1$ we get
\begin{multline}
p(s_1,s_2)\leq k_D(\gamma(s_1),\gamma(s_2))\leq k_D(\gamma(s_1),\gamma(t_1))+k_D(\gamma(t_1),\gamma(s_2))=\\
 p(F(f(s_1)),F(f(t_1)))+p(t_1,s_2)=p(s_1,s_2),
\end{multline}
which finishes the desired claim
\end{remark}

\subsection{Property of uniqueness of right inverses and real geodesics}
Recall that in the paper \cite{Kos-Zwo 2016} the following problem was studied (and solved for strongly linearly convex domains, symmetrized bidisc and the tetrablock): having given a complex geodesic $\varphi$ in the Lempert domain $D$ decide when there is a uniquely determined left inverse $F$. The results below suggest that a kind of a dual problem is interesting, too.

Namely, fix a point $w$ in the Lempert domain $D$. Decide when the extremal function $F:D\to\mathbb D$, $F(w)=0$ for the Carath\'eodory problem for some pair $(w,z)\in D\times D$ (or $(w;X)\in D\times \mathbb C^n$) has only one complex geodesic $f:\mathbb D\to D$ such that $f(0)=p$ with the property that $F$ is a left inverse for $f$. We formulate the property formally below.

We say that the Lempert domain $D$ has {\it uniqueness property of right inverses at $w$} if for any function $F:D\to\mathbb D$ such that $F(w)=0$ there is at most one $f:\mathbb D\to D$ such that $f(0)=w$ and $F\circ f$ is the identity on $\mathbb D$. Note that the uniqueness property of right inverses at $w$ implies the uniqueness of complex geodesics passing through $w$ (both in the standard and infinitesimal case). 

The interest in considering this problem may arise from the following result.

\begin{proposition}
Let $D$ be a Lempert domain, $w\in D$. Then all the real geodesics passing through $w$ are induced by complex ones if and only if $D$ has the uniqueness property of right inverses at $w$.
\end{proposition} 
\begin{proof}
Assume that $f,g:\mathbb D\to D$ are two different complex geodesics such that $f(0)=g(0)=w$ and $F$ is the left inverse to both $f$ and $g$ or $F(f(\lambda))=F(g(\lambda))=\lambda$.
Define the curve:
$\gamma(t):=f(t)$ for $-1<t\leq 0$ and $\gamma(t):=g(t)$ for $1>t\geq 0$.

Then $F(\gamma(t))=t$ and consequently $\gamma$ is a real geodesic that does not come from the complex geodesic. Actually, the non-trivial condition is verified as follows
\begin{multline}
p(t_1,t_2)\leq k_D(f(t_1),g(t_2))\leq k_D(f(t_1),f(0))+k_D(g(0),g(t_2))=\\
p(t_1,0)+p(0,t_2)=p(t_1,t_2)
\end{multline}
for $-1<t_1<0<t_2<1$.

Now we prove the opposite direction. Let $\gamma$ be a real geodesic such that $\gamma(0)=w$. Its left inverse $F$ is a Carath\'eodory extremal for any pair $(\gamma(0),\gamma(t))$. Let $f_t:\mathbb D\to D$, $t\in (-1,1)\setminus\{0\}$ be a complex geodesic with $f_t(0)=p$ and $f_t(t)=\gamma(t)$. Then $F(f_t(0))=0$ and $F(f_t(t))=t$ so the Schwarz Lemma implies that $F(f_t(\lambda))=\lambda$, $\lambda\in\mathbb D$ so the uniqueness of right inverses at $w$ implies that $f_t$ are all equal and consequently $\gamma$ is induced by a complex geodesic.
\end{proof}

\begin{proposition}
Let $D\subset\mathbb C^n$ be a Lempert domain. Assume that $w\in D$.
Let $f,g:\mathbb D\to D$ be complex geodesics passing through $w$ and such that $f(0)=g(0)=w$. 

Then the geodesics $f,g$ have one common left inverse $F:D\to\mathbb D$ (i.e $F\circ f$ and $F\circ g$ are identities) iff $\kappa_D(w;tf^{\prime}(0)+(1-t)g^{\prime}(0))=1$ for any $t\in[0,1]$.
\end{proposition}
\begin{proof} Assume the existence of one left inverse $F$.
Let $h:\mathbb D\to D$ be a complex geodesic for the pair $(w;t\phi^{\prime}(0)+(1-t)\psi^{\prime}(0)=:X_t)$ for some $t\in[0,1]$ with $h(0)=w$. Then we have the following inequalities
\begin{multline}
1=t\kappa_D(w;f^{\prime}(0))+(1-t)\kappa_D(w,g^{\prime}(0))\geq \kappa_D(w;X_t)\geq\\
\kappa_{\mathbb D}(0;(F^{\prime}(w)(h^{\prime}(0)))=|t(F\circ f)^{\prime}(0)+(1-t)(F\circ g)^{\prime}(0)|=1,
\end{multline}
which finishes the proof.

To show the opposite implication let $F$ be a left inverse for the complex geodesic $h:\mathbb D\to D$ with $h(0)=w$, $h^{\prime}(0)=(f^{\prime}(0)+g^{\prime}(0))/2$. Then $|F^{\prime}(0)(f^{\prime}(0))|,|F^{\prime}(0)(g^{\prime}(0))|\leq 1$ and
\begin{multline}
1=\kappa_D(w;h^{\prime}(0))=1/2|F^{\prime}(0)(f^{\prime}(0)+g^{\prime}(0))|\leq \\
1/2(|(F\circ f)^{\prime}(0)|+|(F\circ g)^{\prime}(0)|)\leq 1.
\end{multline}
The equality above holds iff $F\circ f$, $F\circ g$ and $F\circ h$ are the same rotations which finishes the proof.
\end{proof}

Denote by $I_D(p):=\{X\in\mathbb C^n:\kappa_D(p;X)<1\}$ the {\it indicatrix of $D$ with the centre at $p\in D$}).
    

\begin{corollary}\label{corollary:uniqueness-real-geodesics}
Let $D$ be a Lempert domain. Fix $p\in D$. Then the following are equivalent
\begin{itemize}
\item all the real geodesics passing through $p$ are induced by complex geodesics,\\
\item $D$ has uniqueness property of right inverses at $p$,\\
\item $I_D(p)$ is strictly convex and for any $X$ with $\kappa_D(p;X)=1$ there is only one complex geodesic for the pair $(p;X)$.
\end{itemize}
\end{corollary}

Recall that the uniqueness of complex geodesics is satisfied in the following classes (see e. g. \cite{Din 1989}, \cite{Pfl-Zwo 2018}, \cite{Lem 1981}): 
\begin{itemize}
\item bounded strictly convex domains in $\mathbb C^n$,\\
\item tube domains overs bounded strictly convex bases in $\mathbb R^n$,\\
\item strongly linearly convex domains in $\mathbb C^n$, $n>1$.
\end{itemize}

Below we shall see that for $D$ belonging to one of the three classes of domains the property as in Corollary~\ref{corollary:uniqueness-real-geodesics} is satisfied. This follows from the following property.

\begin{proposition}\label{proposition:indicatrix-strictly-convex} Let $D$ be a domain belonging to one of the classes
\begin{itemize}
\item bounded strictly convex domains in $\mathbb C^n$,\\
\item tube domains overs bounded strictly convex bases in $\mathbb R^n$,\\
\item strongly linearly convex domains in $\mathbb C^n$, $n>1$.
\end{itemize}
Fix $p\in D$. Then $I_D(p)$ is strictly convex.
\end{proposition}
\begin{proof} Assume first that $D$ belongs to one of the first two classes of domains. It is sufficient to show that $D$ has uniqueness property for right inverses at $p$. Suppose the opposite. Let $f,g:\mathbb D\to D$ be two different complex geodesics such that $f(0)=g(0)=p$ that have the same left inverse $F$. Then $F$ is the left inverse to all (complex geodesics) $tf+(1-t)g$, $t\in[0,1]$. This is an immediate consequence of the fact that $F(tf(\lambda)+(1-t)g(\lambda))=\lambda$, $t\in[0,1]$, $\lambda\in\mathbb D$, which is a direct consequence of the infinitesimal version of the Schwarz Lemma. Then $k_D(p,tf(\lambda)+(1-t)g(\lambda))=p(0,\lambda)$, $t\in[0,1]$, $\lambda\in\mathbb D$. For almost all $|\omega|=1$  (well-defined) elements $f^*(\omega)$, $g^*(\omega)$ (radial limits) are from $\partial D$ so are the segments $[f^*(\omega),g^*(\omega)]$ as complex geodesics are proper mappings.  This easily leads to a contradiction with the strict convexity of suitable sets.

In the case $D$ is strongly linearly convex we proceed as follows. Let $f:\mathbb D\to D$ be a geodesic with $f(0)=p$ such that $f^{\prime}(0)$ is an element from the open segment lying in the boundary of the indicatrix $I_D(p)$. Making use of the Lempert theory (see \cite{Lem 1981}) we may transform $D$ biholomorphically into a domain $G\subset\mathbb D\times\mathbb C^{n-1}$ so that
the transformed complex gedoesic is given by the formula $\varphi(\lambda)=(\lambda,0^{\prime})$ and $G$ is contained in the unit Euclidean ball $\mathbb B_n$. Then certainly $(1,0^{\prime})$ is lying in the boundary of $I_D(0)$ and our assumption implies that the open segment containing $(1,0^{\prime})$ lies in the boundary of $I_D(0)$. But $I_D(0)\subset \mathbb B_n$ - contradiction.
\end{proof}

\begin{remark} A direct consequence of the corollaries above is that in the case of strictly convex domains, strongly linearly convex ones or tube domains over strictly convex bases we have uniqueness property of right inverses and all the real geodesics are uniquely determined and are induced by the complex ones. 

We should make one more comment. The uniqueness of real geodesics in some class of domains and strict convexity of indicatrices could possibly be contained in some papers, some of the subcases could also be the folklore among the experts; however, we could not find the results as formulated above in the existing literature, we could only find special cases (compare e. g. Lemma 3.3 in \cite{Gau-Ses 2013} for the $C^3$-smooth strongly convex case).
\end{remark}

Note that the above result implies the following generalization of Lemma 5 in \cite{Kos-Zwo 2018}. As the result is not essential in the sequel we do not get deeper into its content (we refer the interested Reader to that paper for recalling the notions used).

\begin{corollary} 
Let $D$ be a bounded strictly convex domain in $\mathbb C^n$.  Assume that two of three $2$-point subproblems of the $3$-
point Pick problem
$D\to\mathbb D$ 
$z_j\to\sigma_j$, $j=1,2,3$ are extremal. Then $z_1, z_2,z_3\in D$ lie on a common complex geodesic.
\end{corollary}

In the case a pseudoconvex balanced domain is biholomorphic to a convex domain we easily get that the domain itself must be convex (use the fact that the indicatrix of the pseudoconvex balanced domain at $0$ is the domain itself). Similar reasoning allows us to conclude from Proposition~\ref{proposition:indicatrix-strictly-convex} the following result that may be seen as a generalization of the Poincar\'e theorem on holomorphic inequivalence of the ball and the polydisc.

\begin{corollary}
Let $D$ be a pseudoconvex balanced domain that is biholomorphic to a strictly convex bounded domain. Then $D$ is strictly convex.
\end{corollary}

We thank N. Nikolov for drawing our attention to the above conclusion.

\section{(Infinitesimal)  complex isometries} For holomorphically invariant functions we may consider the complex isometries with respect to them. More precisely, 
for bounded domains $D\subset\mathbb C^n$, $G\subset\mathbb C^m$ we call the mapping $F:D\to G$ to be a {\it $d$-isometry} if $d_G(F(w),F(z))=d_D(w,z)$ ($d$ may be quite arbitrary holomorphically invariant function like $c$ or $k$). 
A $C^1$-smooth mapping $F:D\to G$ is a {\it $\delta$-isometry} if $\delta_G(F(w);d_w F(X))=\gamma_D(w;X)$ for any $w\in D$, $X\in\mathbb R^{2n}=\mathbb C^n$, where $d_w$ denotes the (real) Frechet derivative of $F$ at $w$ (and $\delta$ may be quite arbitrary holomorphically invariant metric like $\gamma$ or $\kappa$). Note that if $F:D\mapsto G$ is $C^1$-smooth and $F$ is a $k$-isometry (respectivley, $c$-isometry) then $F$ is a $\kappa$-isometry (respectively, $\gamma$-isometry). 

 In the case of Lempert domains where the holomorphically invariant functions coincide we may omit the letter $d$ or $\delta$ and we may speak about ({\it infinitesimal}) {\it complex isometries}.

Below we formulate a result going beyond the class of Lempert domains so we denote $I_D^{\delta}(w):=\{X\in\mathbb C^n:\delta_D(w;X)<1\}$. Recall that the Carath\'eodory indicatrix $I_D^{\gamma}(w)$ is a convex balanced domain. In the case of Lempert domains we certainly denote the indicatrix without the superscript '$\gamma$'. 

The obvious consequence of the just introduced definition is the following.

\begin{proposition} Let $F:D\to G$ be a $\delta$-isometry. Then $d_wF(I_D^{\delta}(w))=I_G^{\delta}(F(w))\cap d_wF(\mathbb C^n)$.
\end{proposition}
\begin{remark} Note that in the case the domains $D$ and $G$ being Lempert and of equal dimension
the real linear mapping (isomorphism) $d_wF$ from the previous result preserves the points of strict convexity of the boundary of the indicatrix -- note that this shows some kind of Poincar\'e theorem on biholomorphic inequivalence of the ball and polydisc.
\end{remark}

\section{The case of the domain $\triangle$}

Recall that the {\it diamond} is the domain $\triangle:=\{z\in\mathbb C^2:|z_1|+|z_2|<1\}$.

All complex geodesics $f=(f_1,f_2)$ of $\triangle$ are of the following form $f_j(\lambda)=a_j\left(\frac{\lambda-\alpha_j}{1-\overline{\alpha_j}\lambda}\right)^{r_j}\left(\frac{1-\overline{\alpha_j}\lambda}{1-\overline{\alpha_0}\lambda}\right)^2$ where $a_j\in\mathbb C$, $r_j\in\{0,1\}$, $\alpha_j\in\overline{\mathbb D}$ but if $r_j=1$ then $\alpha_j\in\mathbb D$ and the relations as in \cite{Jar-Pfl-Zei 1993} (see also \cite{Gen 1987}) are satisfied. For the given geodesic $f$ denote $A:=\{j\in\{1,2\}:r_j=1\}$.

\subsection{Left inverses and the formula for $\kappa_{\triangle}$}
By elementary calculations, we can see that functions $z_1+\omega z_2$ where $|\omega|=1$ are left inverses to (all) geodesics of $\triangle$ for which $A=\{1,2\}$. 
Moreover, the functions $z_1+\omega z_2$ are Carath\'eodory extremal (i. e. the supremum in the definition of the Carath'eodory distance is attained) for the pairs $(\tau t,0)$ and $(0,-\overline{\omega} \tau s)$ for $t,s\in[0,1)$ so the above easily gives the following formula for $k_{\triangle}$
\begin{equation}
    k_{\triangle}((z_1,0),(0,z_2))=p(-|z_1|,|z_2|),\; z_1,z_2\in\mathbb D.
\end{equation}

As the indicatrix $I_{\triangle}(0)=\triangle$ is not strictly convex we easily construct a real geodesic not coming from a complex one.
\begin{example}
The curve $\gamma$ defined by the formula
$\gamma(t):=(t,0)$, $t\in(-1,0]$, $\gamma(t):=(0,t)$, $t\in[0,1)$ is a real geodesic in $\triangle$ with the left inverse $z_1+z_2$. 
\end{example}

The mappings of the type $z\to (z_1^{p_1},z_2^{p_2})$ where $p_1,p_2\in\{1,2\}$ give the branched covering (proper holomorphic mappings of multiplicity $p_1p_2$) between $\mathcal E(p_1/2,p_2/2):=\{z\in\mathbb C:|z_1|^{p_1}+|z_2|^{p_2}<1\}$ and $\mathcal E(1/2,1/2)=\triangle$ and give the inequality
\begin{equation}
\kappa_{\mathcal E(p_1/2,p_2/2)}(z;X)\geq \kappa_{\triangle}((z_1^{p_1},z_2^{p_2});(p_1z_1^{p_1-1}X_1,p_2z_2^{p_2-1}X_2)).
\end{equation}
The inequality above becomes equality in the case the geodesic in $\mathcal E(p_1/2,p_2/2)$ for the pair $(z;X)$ (denote it by $f$) has the property that $f_j$ has no zero in $\triangle$ for $p_j=2$. This follows directly from the form of complex geodesics in convex complex ellipsoids.

On the other hand if the complex geodesic in $\triangle$ for the pair $(z;X)$ has zeroes for both coordinates (i. e. $r_1=r_2=1$) then we get the formula:
\begin{equation}
\kappa_{\triangle}(z;X)=\sup\left\{\frac{|X_1+\omega X_2|}{1-|z_1+\omega z_2|^2}:|\omega|=1\}\right\}=:m_{\triangle}(z;X).
\end{equation}

Consequently we get the following (almost) effective formula for the Kobayashi-Royden metric of $\triangle$.
Namely,  $\kappa_{\triangle}(z;X)$ is the maximum of the value $m_{\triangle}(z;X)$ and the minimum of the following values (excluding below the expressions when $z_j=0$ in the denominator)
\begin{itemize}
\item $\kappa_{\mathcal E(1/2,1)}\left(\left(z_1,\sqrt{z_2}\right);\left(X_1,\frac{X_2}{2\sqrt{z_2}}\right)\right)$,\\
\item $\kappa_{\mathcal E(1,1/2)}\left(\left(\sqrt{z_1},z_2\right);\left(\frac{X_1}{\sqrt{z_1}},X_2\right)\right)$,\\
\item $\kappa_{\mathbb B_2}\left(\left(\sqrt{z_1},\sqrt{z_2}\right);\left(\frac{X_1}{2\sqrt{z_1}},\frac{X_2}{2\sqrt{z_2}}\right)\right)$.
\end{itemize}
As the formulas for $\kappa_{\mathbb B_2}$ and $\kappa_{\mathcal E(1,1/2)}$ are well known (for the latter see \cite{Bla-Fan-Kle-Kra-Ma 1992}) we have some quisieffective formula for $\kappa_{\triangle}$.


\subsection{Smooth Kobayashi isometries of $\triangle$}
Let us recall that the holomorphic or antiholomorphic isomorphisms of the domain are trivially Kobayashi isometries in the domains. The problem whether Kobayashi isometries or embeddings must be (anti)holomorphic has been intensively studied (see e. g. \cite{Kuc-Ray 1988}, \cite{Zwo 1993}, \cite{Zwo 1995}, \cite{Ses-Ver 2006}, \cite{Mah 2012}, \cite{Gau-Ses 2013}, \cite{Ant 2017}, \cite{Lemm 2022}, \cite{Kim-Seo 2022}, \cite{Edi 2019}, \cite{Edi 2023}). The example of the bidisc and the mapping $\mathbb D^2\owns z\to(z_1,\overline{z_2})\in\mathbb D^2$ shows that the Kobayashi isometries may be of different form. There is also a conjecture by Gaussier-Seshadri (\cite{Gau-Ses 2017}) that in most cases of the convex domains the isometries must be (anti)holomorphic. The most positive result in this direction is due to Edigarian (\cite{Edi 2019}) where the author proved that this is really the case for strictly convex domains. Below we show that in $\triangle$ the result is also valid which suggests that the above mentioned conjecture may be true. Very recently A. Edigarian showed that the same result holds for the symmetrized bidisc, that though not biholomorphic to a convex domain, is a Lempert domain and thus it is an important example that is studied in the complex analysis. There is some (non-obvious) similarity between the domain $\triangle$ and $\mathbb G_2$. Namely, the indicatrix at zero is (up to a muliplication of coordinates) equal to $\triangle$.


\begin{proof}[Proof of Theorem~\ref{theorem:isometry}] We give the proof in several steps. Let $F$ be as in the theorem.

{\bf Step 1.} $F(0,0)=(0,0)$. 

Suppose that $F(0,0)=(z_1,z_2)\in\triangle$, $z_1\neq 0$. The explicit formulas for the complex geodesics allow us to find $X_0$ with $\kappa_{\triangle}(z;X_0)=1$ such that for all $X$ from some neighbourhood $U$ of $X_0$ and such that $\kappa_{\triangle}(z;X)=1$ the complex geodesic $f$ for the pair $(z;X)$ is such that $f_1$ never vanishes on $\mathbb D$. We already know that in such a situation we have the equality for $X$ as above 
\begin{equation}
1=\kappa_{\triangle}(z;X)=\kappa_{\mathcal E(1,1/2)}\left((\sqrt{z_1},z_2);\left(\frac{X_1}{2\sqrt{z_1}},X_2\right)\right).
\end{equation}
Recall that the indicatrix $I_{\mathcal E(1,1/2)}(w)$ is always strictly convex (this follows from strict convexity of $\mathcal E(1,1/2)$ and Proposition~\ref{proposition:indicatrix-strictly-convex}). Consequently, as the non-empty open set (in the relative topology of the boundary of the indicatrix) of strictly convex points is preserved under the linear isomorphism $d_{(0,0)}F$ we would get a non-empty open portion of strictly convex points in the boundary of $I_{\triangle}(0,0)=\triangle$ which gives a contradiction.

{\bf Step 2.} $F$ maps the axes $S:=\mathbb D\times\{0\}\cup\{0\}\times\mathbb D$ to the same set. One may therefore conclude that $F$ when restricted to the set $S$ maps (up to a rotation and permutation of variables) points $(z_1,0)$ to $(z_1,0)$ or $(\overline{z_1},0)$ and points $(0,z_2)$ to $(0,z_2)$ or $(0,\overline{z_2})$. In particular, we may assume that $F_1(z_1,0)|=|z_1|$ and $|F_2(0,z_2)|=|z_2|$, $z_1,z_2\in\mathbb D$.

From the isometricity of $F$ and the fact that $F(0,0)=(0,0)$ we get that for all $z\in\triangle$ we have $|F_1(z)|+|F_2(z)|=|z_1|+|z_2|$.

Consider the (real) geodesic $(-1,1)\omega\times\{0\}$, where $|\omega|=1$. Suppose that for some $t_0\in(-1,1)$, $t_0\neq 0$ $F(t_0\omega,0)=t_0(z_1,z_2)\in\triangle$, $z_1z_2\neq 0$, $|z_1|+|z_2|=1$. Note that the geodesic segment $[0,t_0]\omega\times\{0\}$ is uniquely determined in the sense that any point $w\in\triangle$ such that $k_{\triangle}((0,0),t_0(\omega,0))=k_{\triangle}(0,w)+k_{\triangle}(w,(t_0\omega,0))$ must be from $[0,t_0]\omega\times\{0\}$. It easily follows from theisometricity that the real geodesic segment $F([0,t_0]\omega\times\{0\})$ has the same property, which also means that the latter segment comes from a complex geodesic joining $(0,0)=F(0,0)$ with $t_0z$. But such a segment must be the Euclidean segment joining $(0,0)$ with $t_0z$. Then we easily get that the real derivative $d_{(0,0)}F$ maps $(\omega,0)$ to $z$. That gives us however contradiction as the point $(\omega,0)$ is of strict convexity of $\partial\triangle$ whereas the other one (the point $z$) is not.
 
{\bf Step 3.} We show that $|F_1(z_1,z_2)|=|z_1|$ and $|F_2(z_1,z_2)|=|z_2|$, $z\in\triangle$.

One may verify that $k_{\triangle}((z_1,0),(z_1,z_2))=p\left(0,\frac{|z_2|}{1-|z_1|}\right)$, $z\in\triangle$.

Fix for a while $z\in\triangle$. Let $F(z_1,z_2)=(w_1,w_2)$. Isometricity of $F$ applied to $(0,0)$ and $z$ gives that $|w_1|+|w_2|=|z_1|+|z_2|$. On the other hand by considering the function $G(v):=\frac{v_2}{1-\tau v_1}$, where $\tau w_1=|w_1|$, $|\tau|=1$ we get that
\begin{multline}
p\left(0,\frac{|z_2|}{1-|z_1|}\right)=k_{\triangle}((z_1,0),(z_1,z_2))=k_{\triangle}(F(z_1,0),(w_1,w_2))=\\
k_{\triangle}((\omega z_1,0),(w_1,w_2))\geq p(G(\omega z_1,0),G(w_1,w_2))=p\left(0,\frac{w_2}{1-|w_1|}\right).
\end{multline}
Consequently,
\begin{equation}
\frac{|w_2|}{1-|w_1|}\leq\frac{|z_2|}{1-|z_1|}.
\end{equation}
Substituting, $|w_2|=|z_1|+|z_2|-|w_1|$ we get after elementary transformations that $|z_1|\leq|w_1|$. Analoguously we get that $|z_2|\leq|w_2|$, which implies that $|z_1|=|w_1|$ and $|z_2|=|w_2|$ and finishes this step.

{\bf Step 4.} $F$ maps every segment $[(0,0),z]$, where $|z_1|+|z_2|=1$, to a segment $[(0,0),(\omega_1z_1,\omega_2z_2)]$ by the formula $tz\to t(\omega_1z_1,\omega_2z_2)$.

Actually, consider any $|z_1|+|z_2|=1$, $z_j\neq 0$. We already know that $F(tz)=t(\omega_1(t) z_1,\omega_2(t)z_2)$ for some $|\omega_j(t)|=1$ and the set $\{F(tz):t\in(-1,1)\}$ is a real geodesic passing through $(0,0)$. 


Moreover, a left inverse to the real geodesic passing through $(0,0)$ is of the form $\tau_1v_1+\tau_2 v_2$ so for any $t\in(-1,1)$ we have $\omega_1(t)\tau_1z_1+\omega_2(t)\tau_2z_2=1$, so $\omega_j(t)\tau_j=\frac{|z_j|}{z_j}$, which easily finishes the proof.

{\bf Step 5.} Note that it follows from the previous remark that $F(tz)=td_{(0,0)}F(z)$, for any $z\in\overline{\triangle}$, $t\in(-1,1)$ so $F$ a real linear mapping that is determined by its values on the set $S$. Therefore, up to a permutation, conjugation (of both variables) and rotations, it must be either identity or of the form $F(z_1,z_2)=(z_1,\overline{z_2})$, $\in\triangle$. We will show that the later case is not possible.

Actually, suppose that $F$ defined by the latter formula is an isometry. It follows from the form of geodesics that the map $\mathbb B_2\owns z\to (z_1^2,z_2^2)\in\triangle$ preserves the Kobayashi distance for pairs $w,z\in\mathbb B_2$ that are connected by complex geodesics intersecting no axes. It is an elementary observation to see that one may find two non-empty open sets $U_1,U_2$ in $\mathbb B_2$ that are invariant under the mapping $v\to (v_1,\overline{v_2})$ and the geodesics joining points from $U_1$ with that from $U_2$ omit both axes. Then for $w\in U_1$, $z\in U_2$ we get
\begin{multline}
    k_{\mathbb B_2}(w,z)=k_{\triangle}(w_1^2,w_2^2),(z_1^2,z_2^2))=k_{\triangle}((w_1^2,\overline{w_2}^2),(z_1^2,\overline{z_2}^2))=\\
    k_{\mathbb B_2}((w_1,\overline{w_2}),(z_1,\overline{z_2})).
\end{multline}
As the mapping $(w,z)\to k_{\mathbb B_2}^2(w,z)$ is real analytic we get by the identity principle that the mapping $z\to(z_1,\overline{z_2})$ is a Kobayashi isometry of the ball $\mathbb B_2$ -- contradiction.

\end{proof}

\end{document}